\documentclass[preprint,12pt]{elsarticle}

\usepackage{amsmath}
\usepackage{amssymb}
\usepackage{enumitem}
\usepackage{url}
\usepackage[vlined,linesnumbered,titlenumbered,algo2e,ruled]{algorithm2e}

\newtheorem{theorem}{Theorem}

\newtheorem{proposition}{Proposition}
\newtheorem{definition}{Definition}

\newenvironment{proof}{\par \noindent \textit{Proof}.}{\hfill $\Box$\newline}

\journal{Journal of Algebra}

\begin{document}

\begin{frontmatter}

\title{Polynomial recognition of cluster algebras of finite type}

\author{ {\bf {\large Elis\^{a}ngela Silva Dias,
 Diane Castonguay}}}

\address{{\small Instituto de Inform\'{a}tica, Universidade Federal de Goi\'{a}s -- UFG} \\
 {\small Alameda Palmeiras, Quadra D, Campus Samambaia -- Goi\^{a}nia -- Goi\'{a}s -- Brazil} \\
 {\small CEP 74001-970 -- Phone numbers: +556281180095/+556281228259}\\
 {\small E-mail: \{elisangela,diane\}@inf.ufg.br}}

\begin{abstract}
Cluster algebras are a recent topic of study and have been shown to be a useful tool to characterize structures in several knowledge fields. An important problem is to establish whether or not a given cluster algebra is of finite type. Using the standard definition, the problem is infeasible since it uses mutations that can  lead to an infinite process. Barot, Geiss and Zelevinsky (2006) presented an easier way to verify if a given algebra is of finite type, by testing that all chordless cycles of the graph related to the algebra are cyclically oriented and that there exists a positive quasi-Cartan companion of the skew-symmetrizable matrix related to the algebra. We develop an algorithm that verifies these conditions and decides whether or not a cluster algebra is of finite type in polynomial time. The second part of the algorithm is used to prove that the more general problem to decide if a matrix has a positive quasi-Cartan companion is in $\mathcal{NP}$.
\end{abstract}


\end{frontmatter}

\section{Introduction}
\label{sec:introduction}

Matrices can be used to represent various structures arising in combinatorics, including graphs and algebras, such as cluster algebras. The latter can be defined using a directed graph $G(B)$, called quiver, and consequently by an adjcency matrix, where rows and columns represent the vertices and the positive values at positions $(i, j)$ represent the quantity of edges between associated vertices of the graph.

In 2002, Sergey Fomin and Andrei Zelevinsky~\cite{FZ2002} introduced a commutative algebra class called cluster algebras. These algebras have a strong combinatorial structure. They are tools to study questions of dual canonical bases and positivity of semisimple Lie groups. Cluster algebras are defined recursively via commutative algebras with a distinct set of generating variables ({\it cluster variables}) grouped into overlapping subsets ({\it clusters}) of fixed cardinality.

A basic feature of any cluster algebra class is that both the generators and the relationships between them are not given from the start, but are produced by an elementary iterative process of seed mutation. This process is somewhat counterintuitive, in order to encode a universal phenomenon in some way. This may explain the accelerated development of cluster algebra theme in areas such as combinatorics, physics, mathematics (especially geometry), among others as discussed in~\cite{GSV2010,GG2014,MSW2011}. This algebra can be defined using a skew-symmetrizable matrix, as we will see in Section~\ref{sec:finite_type}.

The notion of quasi-Cartan matrices was introduced by Barot, Geiss and Zelevinsky~\cite{BGZ2006}. A quasi-Cartan matrix is a symmetrizable matrix with all entries of the main diagonal equal to 2. The authors show some properties of the matrices from the mathematical point of view. A quasi-Cartan companion is a quasi-Cartan matrix associated with skew-symmetrizable matrix as we will see in Section~\ref{sec:preliminaries}.

One can decide whether a cluster algebra is of finite type (has a finite number of cluster variables) by deciding whether or not the skew-symmetrizable matrix is associated with a cyclically oriented graph and has a positive quasi-Cartan companion. By the Sylvester criterion~\cite{W1976}, a symmetric matrix is positive if all its leading principal submatrices have positive determinant.

In this paper, we present three algorithms of polynomial time complexity. The first one can be used to decide whether or not a quasi-Cartan companion matrix is positive. This is used as a certificate to prove that the problem of setting whether a positive quasi-Cartan companion exists or not belongs to the $\mathcal{NP}$ class of problems. The second algorithm can be used to establish whether or not there exists a positive quasi-Cartan companion of a skew-symmetrizable matrix associated with a cyclically oriented graph. To verify that an oriented graph is cyclically oriented, we will use the algorithm and results presented in~\cite{CD2015}. The last algorithm can be used to decide whether or not a cluster algebra is of finite type. This is used to prove that the problem belongs to the $\mathcal{P}$ class of problems. For more information about $\mathcal{P}$, $\mathcal{NP}$ and $\mathcal{NP}$-complete classes, see~\cite{CLSR2002,S2006}.

The remainder of the paper is organized as follows: Section~\ref{sec:preliminaries} presents the preliminary concepts; Section~\ref{sec:finite_type} defines what is a cluster algebra of finite type and presents some of its properties; In Section~\ref{sec:cyclically_oriented}, we present some properties of cyclically oriented graphs and the general idea of the polynomial algorithm for its determination presented is in~\cite{CD2015}; Section~\ref{sec:cqc_positive} presents a polynomial algorithm for determining if a skew-symmetrizable matrix has a positive quasi-Cartan companion when the associated graph is cyclically oriented; Section~\ref{sec:polynomial} shows that deciding whether a cluster algebra is of finite type is a problem belonging to the class of polynomial problems. Finally, Section~\ref{sec:conclusions} closes the paper with concluding remarks and future work.

\section{Preliminaries}
\label{sec:preliminaries}

Let $n$ be a positive integer, $A, B, C \in \mathrm{M}_n(\mathbb {Z})$ and $D \in \mathrm{M}_n(\mathbb{Q})$. A matrix $A$ is \textit {symmetric} if $A = A^T$, where $A^T$ is the transpose of $A$. A matrix $C$ is \textit {symmetrizable} if $D \times C$ is symmetric for some diagonal matrix $D$ with positive diagonal entries. In this case, the matrix $D \times C$ is the \textit{symmetrization or symmetrized} version of $C$ and the matrix $D$ is the \textit {symmetrizer} of $C$. Note that this definition is equivalent to the one given in~\cite{CM1980} and the matrix $D$ is taken over $\mathbb{Q}$ for the sake of simplicity. A matrix $C$ is \textit{symmetric by signs} if for all $i, j \in \{1, \dots, n\}$, with $i \neq j$, we have $c_{ij} = c_{ji} = 0$ or $c_{ij} \cdot c_{ji} > 0$.

All symmetric matrix are symmetrizable and that all symmetrizable matrices are symmetric by signs.

A matrix $A$ is \textit{skew-symmetric} if $A^T = -A$. Observe that the values of the main diagonal are null. A matrix $B$ is \textit{skew-symmetrizable} if there exists a diagonal matrix $D$ with positive entries such that $D \times B$ is a skew-symmetric matrix. In this case, the matrix $D \times B$ is the \textit{skew-symmetrization} or {\it skew-symmetrized} of $B$ and the matrix $D$ is the \textit{skew-symmetrizer} of $B$. A matrix $B$ is \textit{skew-symmetric by signs} if for all $i, j \in \{1, \dots, n\}$ we have $b_{ii} = 0$ and if $i \neq j$, then $b_{ij} = b_{ji} = 0$ or $b_{ij} \cdot b_{ji} < 0$.

Also observe that all skew-symmetric matrix are skew-symmetrizable and that all skew-symmetrizable matrices are skew-symmetric by signs.

A {\it generalized Cartan} matrix is a symmetric matrix, where its main diagonal entries are equal to 2 and its other values are non-positives. A {\it Cartan matrix} is a symmetrizable matrix with a positive definite symmetrized matrix. Cartan matrices were first introduced by the French mathematician \'{E}lie Cartan. In fact, Cartan matrices, in the context of Lie algebras, were first investigated by Wilhelm Killing, whereas the Killing form is due to Cartan. Positive Cartan matrices represent the basis of the Cartan-Killing classification, see~\cite{FZ2003}.

A symmetrizable matrix is \textit{quasi-Cartan} if all the entries on its main diagonal equal to 2. For a skew-symmetrizable matrix $B$, we will refer to a quasi-Cartan matrix $C$ with $|c_{ij}| = |b_{ij}|$ for all $i \neq j$ as a \textit {quasi-Cartan companion} of $B$.


Given a skew-symmetrizable matrix $B$, we associate $G(B)$ to an oriented graph with vertices $\{1, 2, \ldots, n\}$ and edges $(i, j)$ for each $b_ {ij} > 0$, with $i,j \in \{1, \ldots, n\}$.

Let $G$ be a simple graph, with edge set $E$ and vertex set $V$. Let $n$ be the number of vertices in $V$ and $m$ the number of edges in $E$. Cyclically orientable graphs were introduced in 2006 by Barot, Geiss and Zelevinsky~\cite{BGZ2006}.

A \textit{simple path} is a finite sequence of vertices $\langle v_1, v_2, \dots, v_t \rangle$ such that $(v_i, v_{i+1})\allowbreak \in E$ for $i \in \{1, \dots, t-1\}$ and no vertex is repeated in the sequence, that is,  $v_i\neq v_j$, for $i,j \in \{1, \dots, t\}$ and $i \neq j$. A \textit{cycle} is a simple path $\langle v_1, v_2, \dots, v_t \rangle$ such that $(v_t, v_1)\in E$. We denote a cycle with $t$ vertices by $C_t$.\footnote{Observe that our definition of a cycle, as in~\cite{DCLJ2014}, does not repeat the first vertex at the end of the sequence as usually done by other authors.} A {\it chord} of a cycle is an edge between two vertices of the cycle, that is not part of it. A cycle without chord is a {\it chordless cycle}.

An {\it orientation} of a graph $G(B)$ consists of assigning an order to the endpoints of each of its edges. The orientation of a cycle $\langle v_1, v_2 \ldots,v_t\rangle$ is {\it cyclic} if it receives the orientations $(v_1, v_2), \ldots, (v_{t-1}, v_t), (v_t, v_1)$ or the opposite.

An oriented graph $G(B)$ is {\it cyclically oriented} if any chordless cycle in $G(B)$ is cyclic. It is {\it cyclically orientable} if it admits an orientation in which $G(B)$ is cyclically oriented.

\section{Cluster algebras of finite type}
\label{sec:finite_type}

Cluster algebras are associated with the mutation-equivalence classes of skew-symmetrizable matrices. Recall from~\cite{FZ2002} that, for each matrix index $k$, the mutation in direction $k$ transforms a skew-symmetrizable matrix $B$ into another skew-symmetrizable matrix $B' = \mu_k(B)$, whose entries are given by
\begin{equation}
b'_{ij} =
    \begin{cases}
          -b_{ij}, & \hbox{if $i=k$ or $j=k$;} \\
          b_{ij} + sgn(b_{ik}) \cdot [b_{ik} \cdot b_{kj}]_+, & \hbox{otherwise.}
               \end{cases}
\end{equation}
where we use the notation $[x]_+ = \max(x,0)$, with the convention $sgn(0)=0$ and $sgn(x) = x/|x|$. One can easily check that $\mu_k$ is involutive, implying that the repeated mutations in all directions give rise to the mutation-equivalence relation on skew-symmetrizable matrices. Each of these matrices define an {\it exchange rule}.

For an integer $n$, a {\it cluster algebra} of dimension $n$ is a commutative ring (unitary) without zero divisors, generated in the center of a fixed field $\mathcal{F}$, for a set (possibly infinite) of cluster variables. These variables are not fixed arbitrarily. The set of cluster variables is the union (not disjoint) of subsets of $n$ elements called clusters, which are related by the mutation change. For any cluster $\mathcal{X}$ and any cluster variable $x \in \mathcal{X}$, there is another cluster obtained by substitution of the variable $x$ by another $x'$ related by a binomial relationship as follows: $x \cdot x' = M_1 + M_2$, where $M_1$ and $M_2$ are defined using an exchange rule matrix. All cluster variables are recursively obtained in this way from an initial seed (a cluster with an exchange rule matrix). This process is a {\it seed mutation}. For more information, see~\cite{FZ2002}.

Cluster algebras of finite type are those that have a finite number of cluster variables. They coincide with the famous Cartan-Killing classification of semisimple Lie algebras.

The following criterion for deciding whether a skew-symmetrizable matrix corresponds to a cluster algebra of finite type is presented in~\cite{BGZ2006}.

\begin{theorem} [Barot, Geiss and Zelevinsky \cite{BGZ2006}] \label{teo_1.1} Given a cluster algebra $\mathcal{A}(B)$ associated with a skew-symmetrizable matrix $B$. Let $\mathcal{S}$ be the mutation-equivalence class of $B$. The following are equivalent:
\begin{enumerate}[label=(\alph*)]
	\item \label{eq-1} The cluster algebra $\mathcal{A}(B)$ is of finite type.
	\item \label{eq-2} $\mathcal{S}$ contains a matrix $B'$ such that the Cartan matrix $C$ with off-diagonal entries $c_{ij} = -|b'_{ij}|$ is positive.
	\item \label{eq-3} For every $B' \in \mathcal{S}$ and all $i \neq j$, we have $|b'_{ij} \cdot b'_{ji}| \leq 3$.
	\item \label{eq-4} Every chordless cycle in $G(B)$ is cyclically oriented, and $B$ has a positive quasi-Cartan companion.
\end{enumerate}
Furthermore, the Cartan-Killing type of the Cartan matrix $C$ in~\ref{eq-2} is uniquely determined by $\mathcal{S}$.
\end{theorem}

Mutations are hard to control, so each of the conditions~\ref{eq-2} and~\ref{eq-3} in Theorem~\ref{teo_1.1} is hard to check in general, since there may be a very large (possibly infinite) number of matrices in $\mathcal{S}$. On the other hand, the condition~\ref{eq-4} leads us to a polynomial algorithm as we see in Section~\ref{sec:polynomial}. As with condition~\ref{eq-4}, we divide the verification into two parts: Section~\ref{sec:cyclically_oriented} presents the ideas of the algorithm to decide whether a graph $G(B)$ is cyclically oriented (that was proposed by us in~\cite{CD2015}) and in Section~\ref{sec:cqc_positive} we present two polynomial algorithms that, together, decide whether $B$ has a positive quasi-Cartan companion when $G(B)$ is cyclically oriented.

\section{Cyclically oriented}
\label{sec:cyclically_oriented}

In this section, we will present some properties of chordless cycles, cyclically orientable and cyclically oriented graphs.

Some properties of cyclically orientable graphs are given by Barot, Geiss and Zelevinsky~\cite{BGZ2006}. New characterizations were obtained by Speyer~\cite{S2005} and Gurvich~\cite{G2008} in later works, which enabled the development of algorithms for the recognition of cyclically orientable graphs. 

The proposition below is used in preprocessing  of our algorithm to determine whether a graph is cyclically oriented.

\begin{proposition} [Speyer~\cite{S2005}] \label{prop:LimiteDeArestas}
If $G$ is a cyclically orientable graph with $n$ vertices, then $G$ has at
most $2 \cdot n - 3$ edges.
\end{proposition}


Now we will discuss two-connected components. They are important since any chordless cycle is contained in exactly one of the components. To calculate them, we can use an algorithm based on the ideas of Tarjan~\cite{T1972} and Szwarcfiter~\cite{S1988}, that has time complexity $\mathcal{O}(n^2)$, which is an application of depth-first search (DFS), described in Cormen et al.~\cite{CLSR2002}.

A graph $G$ is \textit{connected} when there exists a path between each pair of vertices of $G$, otherwise $G$ is {\it disconnected}. A {\it connected component} of a graph $G$ is a maximal connected subgraph of $G$. A graph is {\it two-connected} if it is connected and the elimination of at least two vertices is necessary to disconnect it. A single edge is a two-connected graph.

The following theorem deals with cyclic orientability in two-connected components. The idea is used in our algorithms.

\begin{theorem} [Speyer \cite{S2005}] \label{teo1_Speyer}
A graph $G$ is cyclically orientable if and only if all of its two-connected
components also are. A two-connected graph is cyclically orientable if and only if it is either a cycle, a single edge, or of the form $G' \cup C_t$, where $G'$ is a cyclically orientable graph, $C_t$ is a cycle and $G'$ and $C_t$ meet along a single edge. Moreover, if $G=G' \cup C_t$ is any such decomposition of $G$ into a cycle and a subgraph meeting along a single edge, then $G$ is cyclically orientable if and only if $G'$ is.
\end{theorem}

\subsection{Algorithm to determine whether a graph is cyclically oriented}

Based on the results given by Speyer~\cite{S2005}, Dias and Castonguay in~\cite{CD2015} proposed a polynomial algorithm to verify if a graph is cyclically oriented. The algorithm can be slightly modified to verify whether or not a given oriented graph is cyclically oriented. Moreover, in the positive case, they show that the algorithm returns all chordless cycles of the graph.

We now highlight the ideas of the algorithm. It is based on the analysis of each two-connected component of a given graph. Following the idea presented by Speyer~\cite{S2005}, the algorithm identifies chordless cycles induced by cycle $C_t$ given by Theorem~\ref{teo1_Speyer}. This is done in order to reduce each two-connected component to a unique cycle.

Initially, the algorithm verifies if the given graph satisfies Proposition~\ref{prop:LimiteDeArestas}, that is, if the graph has $2 \cdot n - 3$ edges. If not, it returns the decision NO. Next, it finds all two-connected components and then verifies if each component satisfies Proposition~\ref{prop:LimiteDeArestas}. If not satisfied, it returns NO.

After that, the algorithm stores in a queue $Q$ all vertices of degree two for each two-connected component. Observe that all vertices will be added to $Q$ at most once  and if $G$ is cyclically orientable then all vertices will be added.

The algorithm tries, starting with the vertices of queue $Q$, to find and eliminate paths (cycles) in order to reduce the initial two-connected component to a cycle. If it is able to do so, the component is cyclically oriented. If not, it returns NO. This will continue for all components. In the final process, the given graph will be classified as cyclically orientable if all two-connected components receive a cyclically orientable classification; otherwise, the graph is classified as not cyclically orientable. Therefore, it determines in $\mathcal{O}(n^2)$ time whether this graph is cyclically orientable and, if so, it returns the set of all chordless cycles $S$ of $G$. Here, we will consider the set $S$ as a stack. Observe that there is at most $n$ chordless cycles in the whole graph.

Dias and Castonguay~\cite{CD2015} provided an algorithm to verify whether a non-oriented graph is cyclically orientable. Observe that this paper consider soriented graphs, therefore it is necessary to make small modifications in order that the algorithm can be used in Algorithm~\ref{alg:algebra_tipo_finito}. Consider the modified algorithm called $ChordlessCyclesCOd(G)$. If a graph is cyclically oriented it returns the set $S$ of all chordless cycles. For more information about the algorithm, see~\cite{CD2015}.

\section{Positive quasi-Cartan companion}
\label{sec:cqc_positive}

As we saw in Theorem~\ref{teo_1.1}, Barot, Geiss and Zelevinsky~\cite{BGZ2006} show that for a cluster algebra of finite type we have that the associated skew-symmetrizable matrix $B$ must have a positive quasi-Cartan companion $C$.

By the Sylvester criterion~\cite{W1976}, a symmetric matrix is positive if all leading principal submatrices have positive determinant. A {\it leading principal submatrix} is obtained by iteratively removing the last row and the last column of the matrix. Using~\cite{DCD2015}, a symmetrizable matrix is positive if all leading principal submatrices have positive determinant. 

Dias, Castonguay and Dourado~\cite{DCD2015} presented Algorithm~\ref{alg:is_positive} with time complexity $\theta(n^4)$ to decide whether the given matrix $C$ is positive. This algorithm is used as a verifier 
for the general problem of deciding if there exists a positive quasi-Cartan companion. Thus, the problem belongs to the $\mathcal{NP}$ class.

\begin{center}
\begin{minipage}{0.8 \textwidth}
\begin{algorithm2e}[H]
 \LinesNumbered
\small
 \BlankLine
 \KwIn{A $n \times n$ symmetrizable matrix $C$.}
 \KwOut{The response is if the matrix is positive or not.}
 \BlankLine
	
\ForEach{leading principal submatrix $C'$ of $C$}{
	\If{($det(C') \leq 0$)}{
		\Return NO
	}
}

\Return YES

\BlankLine

\caption{$IsPositive(C)$ \label{alg:is_positive}}
\end{algorithm2e}
\end{minipage}
\end{center}

Verifying all quasi-Cartan companions is exponential since if $B$ is a $n \times n$ skew-symmetrizable matrix there are $2^m$ matrices $C$ which are quasi-Cartan companions of $B$. Note that a quasi-Cartan companion of B is specified by choosing the signs of its off-diagonal matrix entries, with the only requirement being that $sgn(c_{ij}) = sgn(c_{ji})$ and $|c_{ij}| = |b_{ij}|$, for $i \neq j$. Dias, Castonguay and Dourado~\cite{DCD2015} showed that if $C$ is symmetric according to the signs, then $C$ is symmetrizable with the same symmetrizer of $B$.

Thus, it is not efficient to test all possible signs for a positive quasi-Cartan companion. However, Barot, Geiss and Zelevinsky~\cite{BGZ2006} showed a sign condition which simplifies the test of the existence of a quasi-Cartan companion without having to assign all possible signs.


\begin{definition}[Sign condition -- Barot, Geiss and Zelevinsky \cite{BGZ2006}] \label{prop_1.4}
Let $B$ be skew-symmetrizable matrix. A quasi-Cartan companion $C$ of $B$ satisfies the sign condition if, for every chordless cycle $C_t$ in $G(B)$, the product $\prod\limits_{(i, j) \in C_t} (-c_{ij})$ over all edges of $C_t$ is negative.
\end{definition}

In our case, the skew-symmetrizable matrix always has at least one quasi-Cartan companion that satisfies the sign condition.

\begin{proposition}[Barot, Geiss and Zelevinsky \cite{BGZ2006}] \label{prop_1.5} Let $B$ be a skew-symmetrizable matrix. If $G(B)$ is cyclically oriented, then $B$ has a quasi-Cartan companion (not necessarily positive) satisfying the sign condition; furthermore, such a quasi-Cartan companion is unique up to simultaneous sign changes in rows and columns.
\end{proposition}

The next proposition shows that it is enough to check the quasi-Cartan companion satisfying the sign condition.

\begin{proposition}[Barot, Geiss and Zelevinsky \cite{BGZ2006}] \label{prop_1.4}
To be positive, a quasi-Cartan companion $C$ of a skew-symmetrizable matrix $B$ must satisfy the sign condition.
\end{proposition}

Observe that a quasi-Cartan matrix that satisfies the sign condition is not necessarily positive.

The following proposition shows that, for our purpose, it is enough to verify the positivity of one quasi-Cartan companion that satisfies the sign condition. The proposed algorithm to verify if a skew-symmetrizable matrix has a positive quasi-Cartan matrix is based on finding a quasi-Cartan matrix that satisfies the sign condition. 

\begin{proposition} \label{prop_sign_condition}
Let $C$ be a quasi-Cartan companion of a skew-symmetrizable matrix $B$ which satisfies the sign condition. If $G(B)$ is cyclically oriented, then $B$ has a positive quasi-Cartan companion if and only if $C$ is positive.
\end{proposition}

\begin{proof}
Suppose that $G(B)$ is cyclically oriented. Clearly, if $C$ is positive, then $B$ has a positive quasi-Cartan companion. On the other hand, if $B$ has a positive quasi-Cartan companion, say $C'$, then by Proposition~\ref{prop_1.4}, $C'$ satisfies the sign condition. By Proposition~\ref{prop_1.5}, $C$ is obtained from $C'$ by simultaneous sign changes in rows and columns. Therefore, there exists a diagonal matrix $X=(x_{ij})$ with $x_{ii} \in \{-1,1\}$ such that $C=XC'X$. Since $det(C)=det(X)\cdot det(C') \cdot det(X)$ and $det(X) \in \{-1,1\}$, we have that $C$ is positive.
\end{proof}

Next we present an algorithm that defines an attribution of sign
for a quasi-Cartan companion of $B$. In the algorithm, the obtained matrix $C$ will satisfy the sign condition. Later, we verify the positivity  of $C$. By Proposition~\ref{prop_sign_condition}, the matrix $B$ has a positive quasi-Cartan companion if and only if $C$ is positive.

Observe that the stack $S$ contains all chordless cycles obtained by the Algorithm $ChordlessCyclesCOd(G)$ and the set $T$ contains all two-connected components of $G$ that are single edges. 


Based on~\citep{CD2015}, one can see that the cardinality of $S$ is at most $n$. Therefore, the time complexity of Algorithm~\ref{alg:positive_companion} is $\mathcal{O}(n^4)$, due Line~\ref{pc-17}. The rest of the algorithm is $\mathcal{O}(n^2)$.
To prove the correctness of Algorithm~\ref{alg:positive_companion}, we first show that the function $sgn()$ is well defined.

\begin{proposition} \label{prop:well_defined}
The function sign is well defined, that is, $sgn(edge) \in \{1, -1\}$ for all $edge \in E$.
\end{proposition}

\begin{proof}
The line numbers in this proof refer to lines of Algorithm~\ref{alg:positive_companion}. First, we prove that in Line~\ref{pc-13}, $I \neq 0$. We can assume that $G$ is a two-connected graph, since the algorithm acts on each two-connected component. The first cycle is a cycle without a defined edge. Therefore, the first edge does not satisfy the condition in Line~\ref{pc-9} and enters on Line~\ref{pc-11} ($I = 1$). For any other cycles $C_t$, we have $G = G' \cup C_t$ such that only edges of $G'$ are defined. Therefore, the first or second edge of $C_t$ is not defined ($I=1$ or $I=2$). 

Since the graph $G$ is two-connected, for each connected component we have a single edge or the edge is part of a cycle. If $G$ is a single edge, then it receives 1 at Line~\ref{pc-2}. If not, the edge receives 1 at Line~\ref{pc-13} or ``-prod'' at Line~\ref{pc-14}. Therefore, $sign$ is well-defined.
\end{proof}

\begin{center}
\begin{minipage}{1 \textwidth}
\begin{algorithm2e}[H]
\SetKwFunction{cycle}{ChordlessCycles}
\small
 \BlankLine
 \KwIn{A cyclically oriented graph $G$, a skew-symmetrizable matrix $B$ associated to $G$, a stack $S$ of chordless cycles and a set $T$ of single edges (two-connected components of $G$).}
 \KwOut{The response whether or not there exists a positive quasi-Cartan companion.}
 \BlankLine

  \BlankLine
 {\bf initialize} $sgn(x_i, x_j) \leftarrow 0$ for all $i,j \in \{1, 2, \ldots, n\}$\label{pc-1}\\
 {\bf initialize} $sgn(edge) \leftarrow 1$ for all single edges in $T$\label{pc-2}

\While{$(S \neq \varnothing)$}{\label{pc-3}
	{\bf remove} an element $c$ of $S$ \tcc{$c = \langle x_1, \ldots, x_t \rangle$}\label{pc-4}
	
	$x_{t+1} \leftarrow x_1$\label{pc-5}\\
	$I \leftarrow 0$\label{pc-6}\\
	$prod \leftarrow 1$\label{pc-7}\\
	
	\ForEach{$i \in \{1, \ldots, t\}$}{\label{pc-8}
		\eIf{$(sgn(x_i, x_{i+1}) \neq 0)$}{\label{pc-9}
			$prod \leftarrow prod \cdot sgn(x_i, x_{i+1})$\label{pc-10}\\
		}
		{
			\eIf{$(I = 0)$}{\label{pc-11}
				$I \leftarrow i$\label{pc-12}\\
			}
			{
				$sgn(x_i, x_{i+1}) \leftarrow 1$\label{pc-13}
			}
		}
	}
	$sgn(x_I, x_{I+1}) \leftarrow - prod$\label{pc-14}
}

 \BlankLine

 		{\bf initialize} $c_{ij} \leftarrow |b_{ij}| \cdot sgn(x_i, x_j)$ for all $i,j \in \{1, 2, \ldots, n\}$\label{pc-15}\\
	{\bf initialize} $c_{ii} \leftarrow 2$ for all $i \in \{1, 2, \ldots, n\}$\label{pc-16}\\

\eIf{$(IsPositive(C))$}{\label{pc-17}
	\Return YES\label{pc-18}
}
{
	\Return NO\label{pc-19}
}

\BlankLine

\caption{$PositiveCompanion(G,B,S)$ \label{alg:positive_companion}}
\end{algorithm2e}
\end{minipage}
\end{center}

We show next that the construction of $C$ gives a quasi-Cartan companion that satisfies the sign condition.

\begin{theorem} \label{teo_sign}
The $n \times n$ matrix $C$ defined below satisfies the sign condition. 
$$c_{ij} =
\begin{cases}
	2, & \mbox{ if } i=j\\
	sgn(x_i, x_j) \cdot |b_{ij}|, & otherwise.
\end{cases}$$
\end{theorem}

\begin{proof}
Recall that $S$ is composed of all chordless cycles of $G$. It follows from Proposition~\ref{prop:well_defined} that the product of each cycle is equal to $-prod^2$.
\end{proof}


\begin{theorem}
The Algorithm~\ref{alg:positive_companion} is correct.
\end{theorem}

\begin{proof}
Follows from \cite{CD2015}, Theorem~\ref{teo_sign} and Proposition~\ref{prop_sign_condition}.
\end{proof}

\section{Recognizing cluster algebras of finite type is in $\mathcal{P}$}
\label{sec:polynomial}

Since the algoritm $ChordlessCyclesCOd(G)$ checks whether or not a graph is cyclically oriented and the algorithm $PositiveCompanion(G,B,S)$ verifies whether or not the matrix has a positive quasi-Cartan companion, both in polynomial time complexity, it follows from Theorem~\ref{teo_1.1} that deciding whether or not a cluster algebra is of finite type belongs to the class of polynomial problems ($\mathcal{P}$).

\begin{center}
\begin{minipage}{0.97 \textwidth}
\begin{algorithm2e}[H]
\SetKwFunction{cyclically}{ChordlessCyclesCOd}
\SetKwFunction{companion}{PositiveCompanion}
\small
 \BlankLine
 \KwIn{A skew-symmetrizable matrix $B$ that defines a cluster algebra $\mathcal{A}(B)$, and the respective graph $G$.}
 \KwOut{The response that the cluster algebra $\mathcal{A}(B)$ is of finite type or not.}
 \BlankLine

 $COd, S \leftarrow \cyclically{G}$\\

\eIf{$(COd = YES)$}{
	\companion{G,B,S}
}
{
	\Return NO
}

\BlankLine

\caption{$ClusterAlgebraFiniteType(B,G)$ \label{alg:algebra_tipo_finito}}
\end{algorithm2e}
\end{minipage}
\end{center}


Based on the above algorithms, we have that the time complexity of Algorithm~\ref{alg:algebra_tipo_finito} is $\mathcal{O}(n^4)$.
The correctness of Algorithm~\ref{alg:algebra_tipo_finito} follows directly from the correctness of Algorithms~\ref{alg:is_positive} and~\ref{alg:positive_companion}.

\section{Conclusions and future work}
\label{sec:conclusions}

In this paper, we proposed two polynomial algorithms that together prove that deciding if a cluster algebra is of finite type belongs to the class of polynomial problems ($\mathcal{P}$). The first checks whether or not the associated graph is cyclically orientable and, if so, the second verifies whether there exists a positive quasi-Cartan companion of $B$, which represents a cluster algebra.
These two criteria can be evaluated for the determination of finiteness of a cluster algebra due to Theorem~\ref{teo_1.1} presented by Barot, Geiss and Zelevinsky~\cite{BGZ2006}.

The result shown is important because the use of an efficient algorithm to determine if a cluster algebra is of finite type will facilitate research in many areas of the application of cluster algebras.
Another interesting solution of the problem was obtained by A. Seven in~\cite{Seven2007}, in which it is verified whether a cluster algebra is of finite type in terms of ``forbidden minors'' of $B$.

Importantly, we conjecture that, in general, the problem of finding a positive quasi-Cartan companion for a skew-symmetrizable matrix $B$ is an $\mathcal{NP}$-complete problem. This problem will be considered in future work.

\section*{Acknowledgements}
\label{sec:acknowledgements}

We would like to thank the Funda\c{c}\~{a}o de Amparo \`{a} Pesquisa do Estado de Goi\'{a}s (FAPEG) for partial support given to this research and the professor Leslie Richard Foulds, from Universidade Federal de Goi\'{a}s, that gave valuable suggestions to improve the paper.


\end{document}